\documentclass[12pt,A4]{article}
\usepackage{amsmath,latexsym}
\usepackage[psamsfonts]{amssymb}
\usepackage{mathrsfs}
\usepackage{amsthm}
\usepackage[OT2,OT1]{fontenc}
\usepackage[all]{xy}
\usepackage{url}

\date{}
\begin{document}

\newtheorem{theorem}{Theorem}[section]
\newtheorem{corollary}[theorem]{Corollary}
\newtheorem{conjecture}[theorem]{Conjecture}
\newtheorem{lemma}[theorem]{Lemma}
\newtheorem{proposition}[theorem]{Proposition}

\theoremstyle{remark}
\newtheorem{remark}[theorem]{Remark}

\theoremstyle{remark}
\newtheorem{example}[theorem]{Example}

\theoremstyle{remark}
\newtheorem*{ack}{Acknowledgements}

\author{L{\!\!'}udov\'{\i}t Balko and J\'ulius Korba\v s}
\title{A note on the characteristic rank and related numbers}
\maketitle

\begin{abstract} This note quantifies, via a sharp inequality, an interplay between (a) the
characteristic rank of a vector bundle over a topological space $X$, (b) the
$\mathbb{Z}_2$-Betti numbers of $X$, and (c) sums of the numbers of certain partitions of
integers. In a particular context, (c) is transformed into a sum of the readily calculable
Betti numbers of the real Grassmann manifolds.



\end{abstract}

\section{Introduction}\label{sec:introductionandresults}
The {\em characteristic rank} of a smooth closed connected $d$-dimensional manifold $M$ was
introduced, and in some cases also calculated, in \cite{korbas2010} as the largest integer
$k$, $0\leq k\leq d$, such that each cohomology class in $H^j(M;\mathbb Z_2)$ with $j\leq k$
can be expressed as a polynomial in the Stiefel-Whitney classes of $M$; further results can
be found in \cite{balko-korbasAMH2012}. Recently, A. Naolekar and A. Thakur
(\cite{naolekarthakur}) have adapted this homotopy invariant of smooth closed connected
manifolds to vector bundles. Given a path-connected topological space $X$ and a real vector
bundle $\alpha$ over $X$, by the {\em characteristic rank of} $\alpha$, denoted
$\rm{charrank}(\alpha)$, we understand the largest $k$, $0 \leq k \leq \mathrm{dim}_{\mathbb
Z_2}(X)$, such that every cohomology class in $H^j(X;\mathbb Z_2)$, $0 \leq j \leq k$, can be
expressed as a polynomial in the Stiefel-Whitney classes $w_i(\alpha)\in H^i(X;\mathbb Z_2)$;
$\mathrm{dim}_{\mathbb Z_2}(X)$ is the supremum of all $q$ such that the cohomology groups
$H^q(X;\mathbb Z_2)$ do not vanish. We shall always use $\mathbb Z_2$ as the coefficient
group for cohomology, and so we shall write $H^i(X)$ instead of $H^i(X;\mathbb Z_2)$ and
$\mathrm{dim}(X)$ instead of $\mathrm{dim}_{\mathbb Z_2}(X)$. Of course, if $TM$ denotes the
tangent bundle of $M$, then we have $\mathrm{charrank}(TM)=\mathrm{charrank}(M)$. Results on
the characteristic rank of vector bundles over the Stiefel manifolds can be found in
\cite{knt2012}. In addition to being an interesting question in its own right, there are
other reasons for investigating the characteristic rank; one of them is its close relation to
the cup-length of a given space (\cite{korbas2010}, \cite{naolekarthakur}).

This note presents a result on the characteristic rank that may also prove useful in some
non-topological questions (for instance, in the theory of partitions; see \cite{andrewstop}).
More precisely, we quantify (Theorem \ref{th:uboundbnum}), via a sharp inequality, an
interplay between (a) the characteristic rank of a vector bundle, (b) the
$\mathbb{Z}_2$-Betti numbers of base space of this vector bundle, and (c) sums of the numbers
of certain partitions of integers. In a particular context, (c) is transformed into a sum of
the readily calculable Betti numbers of the real Grassmann manifolds $G_{n,k}$ of all
$k$-dimensional vector subspaces in $\mathbb R^n$.

\section{The result and its proof}
\label{relationscharrankversusBettinumbers}

Let $p(a,b,c)$ be the number of partitions of $c$ into \emph{at most} $b$ parts, each $\leq
a$. At the same time, given a subset $A$ of the positive integers, let $p(A,b,c)$ denote the
number of partitions of $c$ into $b$ parts each taken from the set $A$. We recall that
(\cite[6.7]{milnorstasheff}) $p(a,b,c)$ is the same as the number of cells of dimension $c$
in the Schubert cell decomposition of the Grassmann manifold $G_{a+b,b}$, or the same as the
$\mathbb Z_2$-Betti number $b_c(G_{a+b,b})=b_c(G_{a+b,a})$. Of course, for dimensional
reasons, $b_c(G_{a+b,b})=p(a,b,c)=0$ for $c > ab = \mathrm{dim}(G_{a+b,b})$. In addition, we
denote by $p(c)$ the total number of partitions of $c$. Even if not explicitly stated, $X$
will always mean a path-connected topological space.

For obvious reasons, we confine our considerations to those vector bundles having total
Stiefel-Whitney class non-trivial. For a given space $X$, an ordered subset in $\{1, 2,
\dots, \mathrm{dim}(X)\}$, with the least element (denoted by) $\nu$ and the greatest element
(denoted by) $\kappa$, will be denoted by $S_{\nu}^{\kappa}$. If, in addition, $\alpha$ is a
real vector bundle over $X$, then $S_{\nu}^{\kappa}(\alpha)$ will denote any
$S_{\nu}^{\kappa}\neq \emptyset$ such that $w_i(\alpha)=0$ for every positive $i\notin
S_{\nu}^{\kappa}$. In general, there are several possible choices of
$S_{\nu}^{\kappa}(\alpha)$: one can always take $S_{\nu}^{\kappa}(\alpha)=\{1=\nu, 2, \dots,
\mathrm{dim}(X)=\kappa\}$, but if we know, for instance, that $w_3(\alpha)=0$, then we can
also take the set $\{1=\nu, 2, 4, \dots, \mathrm{dim}(X)=\kappa\}$ in the role of
$S_{\nu}^{\kappa}(\alpha)$. Now we can state and prove our result.

\begin{theorem} \label{th:uboundbnum}
Let $\alpha$ be a real vector bundle over a path-connected topological space $X$ such that
$w(\alpha)\neq 1$, and let $\mathrm{charrank}(\alpha)\geq t$ for some $t$. Then, for every
$j\in \{1, \dots, t\}$, we have an inequality,
\begin{align}
b_j(X)\leq \sum\limits_{s=1}^{\lfloor\frac{j}{\nu}\rfloor} p(\{x\in
S_{\nu}^{\kappa}(\alpha)\vert x\leq \mu\},s,j), \label{ineq:uboundbettinum}
\end{align}
where $\mu=\mathrm{min}\{j,\kappa\}$.

In particular, if the set $\{x\in S_{\nu}^{\kappa}(\alpha)\vert x\leq \mu\}$ in
\eqref{ineq:uboundbettinum} is gapless, then \eqref{ineq:uboundbettinum} turns into
\begin{align}
b_j(X)\leq \sum_{\frac{j}{\mu}\leq s \leq \frac{j}{\nu}} b_{j-\nu s}(G_{\mu-\nu+s,s}).
\label{ineq:uboundviaGrassmannmanif}
\end{align}

\end{theorem}

\begin{proof}
Let us fix one of the sets $S_{\nu}^{\kappa}(\alpha)$. If $\mathrm{charrank}(\alpha)\geq t$
for some $t$ then, for every $j\in \{1,2, \dots, t\}$, the $\mathbb Z_2$-vector space
$H^j(X)$ is spanned by all the products of the form
\begin{equation}
w_{\nu}^{i_\nu}(\alpha)\cup \cdots \cup w_{\mu}^{i_{\mu}}(\alpha)\in H^j(X).
\label{generatorsofH^j(X)}
\end{equation}
Since $\nu(i_\nu + \dots + i_{\mu}) \leq \nu i_{\nu}+ \dots + \mu i_{\mu} = j$ (thus $i_\nu +
\dots + i_{\mu}\leq \frac{j}{\nu}$), the number of generators of the form
\eqref{generatorsofH^j(X)}, that is, an upper bound for $b_j(X)$, is the number of partitions
of $j$ into at most $\lfloor\frac{j}{\nu}\rfloor$ positive parts each taken from the set
$\{x\in S_{\nu}^{\kappa}(\alpha)\vert x\leq \mu\}$. In other words, this upper bound is
$\sum\limits_{s=1}^{\lfloor\frac{j}{\nu}\rfloor} p(\{x\in S_{\nu}^{\kappa}(\alpha)\vert x\leq
\mu\},s,j)$ as was asserted.

In order to transform \eqref{ineq:uboundbettinum} into \eqref{ineq:uboundviaGrassmannmanif}
when the set $\{x\in S_{\nu}^{\kappa}(\alpha)\vert x\leq \mu\}$ in
\eqref{ineq:uboundbettinum} is gapless, it suffices to know that
\begin{equation}
\sum\limits_{s=1}^{\lfloor\frac{j}{\nu}\rfloor} p(\{\nu, \nu +1, \nu +2, \dots, \mu\},s,j) =
\sum\limits_{s=1}^{\lfloor\frac{j}{\nu}\rfloor} p(\mu-\nu,s,j-\nu s);
\label{formulaviarestrictedpartit} \end{equation} this equality is verified by the following
elementary considerations.

Let $P(j)_{\{\nu,\nu+1,\ldots,\mu\}}^x$ be the set of partitions of $j$ into $x$ positive
parts each taken from the set $\{\nu,\nu+1,\ldots,\mu\}$ and let
$P(l)_{\{1,2,\ldots,\mu-\nu\}}^{\leq x}$ be the set of partitions of $l$ $(l\geq 0)$ into at
most $x$ parts each taken from the set $\{1,2,\ldots,\mu-\nu\}$ (the set
$P(0)_{\{1,2,\ldots,\mu-\nu\}}^{\leq x}$ for all $x>0$ and $\mu-\nu > 0$ contains just one
element, namely the empty partition). Of course, the total number of elements in
$P(l)_{\{1,2,\ldots,\mu-\nu\}}^{\leq x}$ is $p(\mu-\nu,x,l)$.

Each element of $P(j)_{\{\nu,\nu+1,\ldots,\mu\}}^x$ has the form
\begin{align*}
a_1+\cdots+a_{i(\nu)}+a_{i(\nu)+1}+\cdots+a_x,
\end{align*}
where $i(\nu)\geq 0$, $a_1=\cdots =a_{i(\nu)}=\nu<a_{i(\nu)+1}\leq a_{i(\nu)+2}\leq \cdots
\leq a_x$ and $a_1+\cdots+a_{i(\nu)}+a_{i(\nu)+1}+\cdots+a_x=j$. The map
\[
P(j)_{\{\nu,\nu+1,\ldots,\mu\}}^x\longrightarrow P(j-\nu x)_{\{1,2,\ldots,\mu-\nu\}}^{\leq
x},
\]
\[
a_1+\cdots+a_{i(\nu)}+a_{i(\nu)+1}+\cdots+a_x\mapsto
(a_{i(\nu)+1}-\nu)+(a_{i(\nu)+2}-\nu)+\cdots+(a_x-\nu)
\]
is bijective; indeed, the inverse map is readily seen to be
\[
P(j-\nu x)_{\{1,2,\ldots,\mu-\nu\}}^{\leq x}\longrightarrow
P(j)_{\{\nu,\nu+1,\ldots,\mu\}}^x,
\]
\[
b_1+b_2+\cdots+b_{x-i} \mapsto \underbrace{\nu+\cdots+\nu}_{i\text{
times}}+(b_1+\nu)+\cdots+(b_{x-i}+\nu).
\]
Thus \eqref{formulaviarestrictedpartit} is verified, and the proof of Theorem
\ref{th:uboundbnum} is finished.
\end{proof}

\begin{example}\label{illustrationofthmandsharpness}
We recall (\cite[Theorem 7.1]{milnorstasheff}) that the cohomology algebra of the infinite
Grassmannian $G_{\infty,k}$ can be identified with a polynomial algebra,
$$H^\ast(G_{\infty,k})=\mathbb Z_2[w_1, \dots, w_k],$$
where $w_i\in H^i(G_{\infty,k})$ is the $i$th Stiefel-Whitney class of the universal
$k$-plane bundle $\gamma_{\infty,k}$. Thus $\mathrm{charrank}(\gamma_{\infty,k})=\infty$ and
we may take $S_{\nu}^{\kappa}(\gamma_{\infty,k})=\{1=\nu,2,\dots, k=\kappa\}$. Since for
$X=G_{\infty,k}$ there are no relations among the generators of the form
\eqref{generatorsofH^j(X)}, inequalities \eqref{ineq:uboundbettinum} and
\eqref{ineq:uboundviaGrassmannmanif} turn into one of the following equalities for any
positive integer $j$:
\begin{equation} b_j(G_{\infty,k})= p(j)=\sum_{s=1}^{j}
b_{j-s}(G_{j-1+s,s}) \textrm{\,\,if\,\,} j\leq k, \label{equalityforjleqk}\end{equation}
\begin{equation} b_j(G_{\infty,k})= \sum\limits_{s=1}^{j} p(\{1,2,\dots,k\},s,j)
=\sum_{s=\lceil\frac{j}{k}\rceil}^{j} b_{j-s}(G_{k-1+s,s}) \textrm{\,\,if\,\,} j > k.
\label{equalityforj>k}\end{equation} In a similar way, one can see that
\eqref{ineq:uboundbettinum} and \eqref{ineq:uboundviaGrassmannmanif} are also sharp for
$X=\tilde G_{\infty,k}$, the infinite oriented Grassmannian.
\end{example}

\begin{ack}The authors thank Professor James Stasheff for useful comments on a version of this
paper. Part of this research was carried out while J. Korba\v s was a member of the research
teams 1/0330/13 and 2/0029/13 supported in part by the grant agency VEGA (Slovakia).\end{ack}

{\small \hspace{2cm} L{\!\!'}. Balko, Faculty of Mechanical Engineering of SUT,

\hspace{2cm}  N\'amestie slobody 17, SK-812 31 Bratislava, Slovakia

\hspace{2cm} (\url{ludovit.balko@gmail.com}) \\

{\small \hspace{2cm} J. Korba\v s, Faculty of Mathematics, Physics, and Informatics,

\hspace{2cm}  Comenius University, Mlynsk\'a dolina,

\hspace{2cm} SK-842 48 Bratislava, Slovakia {\em and}

\hspace{2cm} {Mathematical Institute of SAS,

\hspace{2cm}  \v Stef\'anikova 49, SK-814 73 Bratislava, Slovakia}

\hspace{2cm} (\url{korbas@fmph.uniba.sk})

\end{document}